\documentclass[reqno,A4]{amsart}

\usepackage{amsmath}
\usepackage{amssymb}
\usepackage{amsthm}
\usepackage{enumerate}
\usepackage{url} 

\usepackage{eqlist}
\usepackage{array}

\begin{document}

\theoremstyle{plain}
\newtheorem{theorem}{Theorem}   
\newtheorem{lemma}[theorem]{Lemma}   
\newtheorem{proposition}[theorem]{Proposition}   
\newtheorem{corollary}[theorem]{Corollary}   

\renewcommand{\labelenumi}{{\textnormal{(\roman{enumi})}}}

%%%%%%%%%%%%%%%%%%%%%%%%%%%%%%%%%%%%%%%%%%%%%%%%%%%%%%%%%%%%%%%%%%%%%%%%%%%%%% 

\title[Varieties of $*$-regular rings]{Varieties of $*$-regular rings}

\author[Christian  Herrmann]{Christian Herrmann}

\address{Technische Universit\"{a}t Darmstadt FB4\\Schlo{\ss}gartenstr. 7\\D--64289 Darmstadt\\GERMANY}

\email{herrmann@mathematik.tu-darmstadt.de}


\begin{abstract}
Given a subdirectly irreducible
$*$-regular ring $R$, we show that $R$ is a homomorphic image of a regular $*$-subring
 of an ultraproduct of  the (simple)  $eRe$, $e$ in the minimal ideal of $R$.
For any subdirect product of artinian $*$-regular rings
we construct a unit-regular and $*$-clean extension within its
variety. 
\end{abstract}

\maketitle

\section{Introduction}
Studying (von Neumann) regular and $*$-regular rings
from an Universal Algebra perspective
was introduced by Goodearl, Menal, and Moncasi \cite{good2}
(free regular rings)
and Tyukavkin \cite{tyu}, who showed 
unit-regularity for subdirect products
of countably many  $*$-regular rings which are algebras over a given uncountable field.
These also relate to the open question, raised by Handelman 
\cite[Problem 48]{good}, whether all $*$-regular rings are directly finite or even
unit-regular.

Tyukavkin's method of  matrix limits
has been elaborated by   Micol \cite{Flo}, Niemann
\cite{Nik}, Semenova, and this author \cite{PartI,act},
to deal with $*$-regular rings which  are representable
as $*$-rings of endomorphisms of 
anisotropic  inner product spaces (given by  $\varepsilon$-hermitean
forms).
In particular,  a variety 
(pseudo-inversion being included into the signature)
of $*$-regular rings
is generated by artinians if and only if
its subdirectly irreducible members are representable \cite[Theorem 3.18, Proposition 3.6]{Flo}.
A variant \cite{dfin} of the  method   constructs a $*$-regular preimage of $R$
within a saturated extension of given $*$-regular $R$.

The present note continues to 
study representability, unit-regualrity, and direct finiteness
from an Universal Algebra point of view.
Any variety 
of $*$-regular rings
 is shown
to be generated by its simple members -
the analogous result for modular ortholattices is due to \cite{hr}.
It follows, by \cite[Theorem 10.1]{act}, that
a variety has all subdirectly irreducibles representable if 
 all simple members are representable.
Finally, any  subdirect product of artinian
$*$-regular rings
is shown to have a unit-regular and $*$-clean extension
within its variety.

\section{Preliminaries}\label{pre}
Recall some basic concepts and results from 
\cite{berb,berb2,good}.
The signature of $*$-rings includes $*,+,-,\cdot,0$
with or without unit; in the former case, unit $1$ is included into
the signature.  
 Section \ref{uni}   requires  rings with unit. 
A $\ast$-\emph{ring} $R$ is a ring  having as additional
operation an 
involution $a \mapsto a^*$.

An element $e$ of $R$ is a
\emph{projection}, if $e=e^2=e^\ast$.
$R$ is \emph{proper},
if $aa^*=0$ only for $a=0$.
A proper $*$-ring $R$ is 
  $\ast$-\emph{regular}
if
it is (von Neumann) regular,
that is for any $a\in R$, there is  $x\in R$ such that $axa=a$;
in particular, for any $a,b \in R$ there is $c$ such that $aR+bR=cR$
and for any $a \in R$ there are (unique) projections $e,f$
such that $aR=eR$ and $Ra=Rf$.
The set $P(R)$ of projections is partially ordered by
 \[e \leq f \Leftrightarrow fe=e \Leftrightarrow ef=e.\]
If $e\in P(R)$  then 
$eRe$  is a $*$-regular 
ring with unit $e$, a $*$-subring of $R$ if unit is not considered.
 In case of $*$-rings
with unit, $eRe$ is  
a homomorphic image of the regular $*$-subring  $eRe+(1-e)R(1-e)$ of $R$,
isomorphic to
$eRe \times (1-e)R(1-e)$.

Considering $*$-rings $R$ with unit,
$R$ is \emph{unit-regular} 
if for any $a \in R$ there is invertible $u\in R$
(a \emph{unit quasi-inverse}) such that
$aua=a$. $R$ is $*$-\emph{clean}
if for any $a \in R$ there are an invertible $u$ and a 
projection $e$ in $R$ such that $a=u+e$.
$R$ is \emph{directly finite}
if $xy=1$ implies $yx=1$; this is implied by unit-regularity.

For the remainder of this Section and
Sections~\ref{tyu} and \ref{sim}, $R$ will be a $*$-regular ring with or without unit;
otherwise, unit is included into the signature.
Recall that  for any $a \in R$ 
there is a [Moore-Penrose]
\emph{pseudo-inverse}
(or Rickart relative inverse) $a^+ \in R$, that is 
\[a=aa^+a,\quad a^+=a^+aa^+,\quad(aa^+)^\ast=aa^+,\quad (a^+a)^\ast=a^+a\]
cf. \cite[Lemma 4]{Kap}.
In this case, $a^+$ is uniquely determined by $a$
(and it follows $(a^*)^+=(a^+)^*$). 
Thus, a sub-$*$-ring $S$ of a $*$-regular ring $R$ 
is $*$-regular if and only if it is closed under
the operation $a \mapsto a^+$.
Obviously, pseudo-inversion is 
compatible with surjective homomorphisms and direct products.
Thus,
 including pseudo-inversion into the signature,
$*$-regular rings form a variety
and we speak of varieties of $*$-regular rings. 
$\exists$-varieties of $*$-regular rings
(cf \cite{PartI,act})
are just the $*$-ring reducts of the latter.
From \cite[Proposition 9]{PartI} one has
the following useful property of pseudo-inversion.
\begin{itemize}
\item[(1)] For any $a,e \in R$, $e$ a projection,
if $ae=0$ then $(a^+)^*e=0$ and if $a^*e=0$ then
$a^+e=0$.
\end{itemize}

Observe  that any ideal $I$ of $R$ 
is generated by the set $P(I)$ of projections in $I$,
more precisely
\[(2)\quad I= \bigcup_{e \in P(I)} eR = \bigcup_{e \in P(I)} Re =\{a^*\mid a \in I\}.\]
Therefore, the congruence lattice of $R$ is isomorphic to 
the lattice of ideals of the ring reduct of $R$.
In particular,  $R$ is subdirectly irreducible 
if and only if it admits a unique minimal ideal,
to be denoted by $J$.

Since in a regular ring
the sets of right (left) ideals are closed under  sums,
it follows from (2)
\begin{itemize}
\item[(3)] For any ideal $I$,
given  finite $A\subseteq I$, 
there is $e \in P(I)$ such that $fa=a=af$ for all $a \in A$
and $f\geq e$.
\end{itemize}

A $*$-regular ring   $R$ is \emph{artinian} if there is a finite
bound on the length of chains of projections; equivalently
the ring $R$ is right (left) artinian; in particular, if $S$ is
a $*$-subring, closed under $a \mapsto a^+$, of
  $\prod_{i=1}^n R_i$ with artinian $R_i$ then
$S$ is artinian.
Artinian $*$-regular rings are
semisimple whence  both
unit-regular and $*$-clean
(cf. \cite[Theorem 4.1]{good} and \cite[Proposition 4]{vas}).

\section{Tyukavkin approximation revisited}

In order to have the method available in various contexts,
we define the notion of approximation setup
(both for the case with and without unit) and
also for $*$-$\Lambda$-algebras.
Recall from \cite{act} that the latter are
$\Lambda$-algebras $R$ where $\Lambda$ is a commutative $*$-ring
and $(\lambda a)^*=\lambda^*a^*$
for all $\lambda \in \Lambda$ and $a \in R$. 
Any $*$-ring is a $*$-$\mathbb{Z}$-algebra.
 
An \emph{approximation setup} 
now consists of $*$-regular $*$-$\Lambda$-algebras
$R$ and $T$, $R$ a $*$-$\Lambda$-subalgebra of $T$,
 a set $P$ of projections of $T$,
and a filter $\mathcal{F}$ on $P$ such that $\emptyset \not\in \mathcal{F}$
and  $ \{p\in P\mid p \geq e\}\in\mathcal{F}$  for all $e \in P$.
Moreover, we require the following
\begin{itemize}
\item[(a)] For all $a\in R$, if $ae=0=ea$ for all $e \in P$
then $a=0$.
\item[(b)] For all $a,b \in R$ and $e \in P$ there is $f \in P$, $f \geq e$,
such that $ea=eaf$ and $be=fbe$.
\end{itemize}
Now, define $A_e=eTe$ and $A=\prod_{e \in P}A_e$,
which are $*$-regular $*$-$\Lambda$-algebras. For 
 $a \in R$ and $\alpha= (a_p\mid p \in P) \in A$ define
\[ a \sim
\alpha 
\Leftrightarrow \forall e \in P.\, \exists X \in \mathcal{F}.\,\forall p \in X.\;
e a=e a_p\; \&\; ae=a_pe.\]

\begin{lemma} \label{tyu} 
$S=\{ \alpha \in A\mid \exists a \in R.\: a \sim  \alpha\}$
 is a $*$-regular $*$-$\Lambda$-subalgebra of $A$  and there is a surjective homomorphism $\varphi:S\to R$
such that $\varphi(\alpha)=a$ if and only if $a \sim \alpha$. 
Moreover, for 
the canonical homomorphism $\psi$
from $A$ onto the reduced product $A/\mathcal{F}$
there is a surjective homomorphism $\chi: A/\mathcal{F}\to R$ 
such that $\varphi=\chi \circ \psi$.
\end{lemma}
\begin{proof}
Consider $a\sim \alpha= (a_p\mid p \in P)$ and
$b\sim \beta=(b_p\mid p \in P)$  and $e \in P$. 
Choose $f$ according to (b).
Choose $X_1,Y_1$ in $\mathcal{F}$ witnessing $a\sim \alpha$ for $e$  respectively 
$b\sim \beta$ for $f$;
choose $X_2,Y_2$ in $\mathcal{F}$ witnessing $b\sim \beta$ for $e$  respectively 
$a\sim \alpha$ for $f$.
 Put $Z=X_1\cap Y_1\cap X_2\cap Y_2$. Then one has for all $p\in Z$
\[\begin{array}{ccccccccc}eab &=&eafb &=&eafb_p&=&eab_p&=&ea_pb_p\\
abe&=& afbe&=&a_pfbe&=&a_pbe&=&a_pb_pe. \end{array}\]
 This shows $ab\sim \alpha \beta$.
Closure of $S$ under the other fundamental $*$-$\Lambda$-algebra operations
is even more obvious.
 In case of  unit (as constant) one has
$1 \sim \alpha$, $\alpha=(p\mid p \in P)$ the unit of $A$,
witnessed by $X=\{ p \in P\mid p\geq e\}$ for $e \in P$.
In view of (a), $\varphi$ is a well defined map.
To prove 
surjectivity of $\varphi$, given $a \in R$ put $a_p=pap$ and
$\alpha=(a_p\mid p \in P)$. Now given $e$,
choose $f$ according to (b) with $a=b$,
and $X=\{p \in P\mid p \geq f\}$; then
$ea=eap =epap=ea_p$ and,
similarly, $ae=a_pe$  for all $p \in X$. This proves
$a \sim \alpha$.

To prove regularity of $S$, having $\operatorname{im} \varphi=R$ regular,
 it suffices to
show that ${\sf ker} \varphi$ is regular
(see \cite[Lemma 1.3]{good}). As in the proof of
Assertion 22 in \cite{PartI} we will show that 
${\sf ker} \varphi$ is closed under pseudo-inversion in $A$. 
Consider 
$0 \sim \alpha =(a_p\mid p \in P)$. For any $e \in P$
there is $X\in \mathcal{F}$ such that $a_pe=0=a_p^*e$
for all $p \in X$ and
  by (1) it follows $(a_p^+)^*e=0=a_p^+e$
for all $p \in X$. Thus,  $0 \sim \alpha^+$.
The last claim folllows from ${\sf ker} \psi \subseteq {\sf ker}\varphi$. 
\end{proof}

\section{Simple generators}\label{sim}

\begin{lemma}\label{l1}
For subdirectly irreducible $R$  and
 any $0\neq a\in R$  there is $e \in P(J)$, $J$ the minimal ideal of $R$,
 such that $eae\neq 0$; also, for any nonzero $e \in P(J)$, $eRe$ is a simple
$*$-regular ring with unit $e$.
\end{lemma}
\begin{proof} 
As $R$ is subdirectly irreducible,
for any $a \neq 0$ the ideal $I$ generated by 
$a$ contains $J$;
that is, for any $f \in P(J)$ one has 
 $f=\sum_i r_ias_i$
 whence $f=\sum_i fr_ias_if$ for suitable $r_i,s_i \in R$. 
Now, choosing $f\neq 0$,
by (3) there is $e\in P(J)$
such that $fr_ie=fr_i$  and  $ es_if=s_if $ for all $i$
whence $eae \neq 0$. 
By the same token,
for  $0\neq a \in eRe$, where $0\neq e \in P(J)$,
one has   $e=\sum_i r_ias_i$ 
 whence $e= \sum_i er_ieaes_ie$
is in the ideal of $eRe$ generated by $a$.
\end{proof}

\begin{theorem}\label{theorem} 
Any   subdirectly irreducible
$*$-regular ring $R$ $($with or without unit$)$ 
 is a homomorphic image of a $*$-regular $*$-subring
 of an ultraproduct of $*$-rings $eRe$, where $e$
is a nonzero projection in the minimal ideal $J$.
In particular, any variety of $*$-regular rings is generated by its
simple members.
\end{theorem}

\begin{proof}
We apply Lemma~\ref{tyu} with $A= \prod_{e\in P}$
where $P$ is a  cofinal subset  of the set $P(J)$ of projections
in the minimal ideal $J$ such that  $0 \not\in P$.
Since the sets $\{p \in P\mid p\geq e\}$, $e\in P(J)$,
form a filter base, 
there is an ultrafilter $\mathcal{F}$ 
on $P$ containing all these. 
This provides an approximation setup since
conditions  (a) and (b) are satisfied in view of Lemma~\ref{l1} and (3).
\end{proof}

Let 
$C$ be the center of the subdirectly irreducible
$*$-regular rings  $R$.
and $A$ the directed union of the 
$gRg+C(1-g)$ where $C$ is the center of $R$.
Consider $R$ with designated subset $A$, denoted as $(R;A)$.
In analogy to the proof of  \cite[Theorem 10.1]{act}
(where $A=\hat{J}(V_F)$)   
choose  an $\omega$-saturated elementary extension $(\hat{R};\hat{A})$ of $(R;A)$ (such an extension  exists, cf. \cite[Corollary 4.3.1.4] {CK}).
Also, choose  $S=\{a \in \hat{A}\mid \exists r \in R.\, a \sim r\}$
where $a \sim r$ if $ae =re$ and
$a^*e=r^*e$ for all $e \in P(J)=:J_0$.  

\begin{lemma}
Then $S$ is a $+$-regular
 $*$-subring of $\hat{A}$ and
$a \mapsto r$ for $a \sim r$ a homomorphism $\varphi$ of $S$ onto $R$.
\end{lemma}
\begin{proof}
This follows as   Claims 1--3 in the proof of \cite[Theorem 10.1]{act},
if  reference to Propositions 2 and 4.4(iv) in \cite{act}
is replaced by that to Lemma \ref{l1} and (3).
Moreover, in view of (1), $\ker \varphi$
is closed under pseudo-inversion in $\hat{R}$.
By \cite[Lemma 1.3]{good} it follows that 
 $S$ is regular, whence $*$-regular.
\end{proof}

\section{Unit-regular extensions}\label{uni}

\begin{theorem}\label{subd}
Every subdirect product $R$ of artinian $*$-regular rings 
has a unit-regular and $*$-clean extension within its variety.
\end{theorem}
The same holds for each first order $\Pi_1$- sentence
(that is, prenex with quantification of the form $\forall \ldots \forall \exists
\ldots \exists$) which is valid in all artinian $*$-regular rings.
\begin{proof}
By hypothesis, there are ideals $I_k$ of $R$, $k \in K$,
such that each $R/I_k$ is artinian and $\bigcap_{k \in K} I_k=0$.
The first order structure
 $(R;I_k(k\in K))$ 
has a  $\omega$-saturated elementary extension
 $(\tilde{R},\tilde{I}_k(k\in K))$.
In particular, 
each  $*$-regular ring $\tilde{R}/\tilde{I}_k$ is artinian
(having the same finite bound on the length of chains of projections
as does $R/I_k$).
For $F\subseteq K$ put 
$\tilde{I}_F=\bigcap_{k\in F}\tilde{I}_k$ and 
$\tilde{R}_F:=\tilde{R}/\tilde{I}_F$.
It follows that, for any finite $F \subseteq K$
the subdirect product $\tilde{R}_F$
is also artinian, whence unit-regular.
On the other hand, $\tilde{R}_K$ 
is a subdirect product of artinians and
 $R$ embeds into $\tilde{R}_K$,
canonically, since $\tilde{I}_k \cap R=I_k$ for all $k \in K$.

Now, fix $r \in R$ and consider the following set $\Sigma$ of formulas
with parameter $r$ and variables $x,y$
\[  \{I_k(xy-1)\; \&\;I_k(yx-1)\; \&\;I_k(rxr-r)  \mid k \in K\}. \]
Given finite $\Sigma_0 \subseteq \Sigma$ there is
finite $F  \subseteq K$ such that predicate symbols occurring in $\Sigma_0$
have $ k \in F$. Thus, $\Sigma_0$ is satisfied substituting,
for $x,y$, elements
$u,s \in \tilde{R}$ witnessing that $r+\tilde{I}_F$
has unit quasi-inverse $u$ in $\tilde{R}_F$.
By saturation there are $u,s \in \tilde{R}$
satisfying $\Sigma$; that is, $u$ is a unit quasi-inverse of $r$
in $\tilde{R}_K$. By the same  approach one
obtains unit $v$ and projection $e$ in $\tilde{R}_K$
such that $r=v+e$. 

To summarize, any subdirect product $R_n$
of artinians has an extension $R_{n+1}$, within the variety of $R_n$,
which is  a subdirect product of artinians and 
such that for every $r \in R_n$ there
are units $u,v$ and projection $e$ in $R_{n+1}$
with $r=rur$ and $r=v+e$. Thus, starting with $R=R_0$,
the directed union $\bigcup_{n<\omega}R_n$ is the required 
extension.
\end{proof}

\begin{corollary}
Within any variety $\mathcal{V}$ generated by artinian 
$*$-regular rings, every member is a homomorphic image
of a $*$-regular $*$-subring of some unit-regular and
$*$-clean member of $\mathcal{V}$.
\end{corollary}
\begin{proof} For any variety, given a class
$\mathcal{G}$ of generators, free algebras are
subdirect products of subalgebras of members of
$\mathcal{G}$.
\end{proof}

\section{Open problems}

Let $\mathcal{A}$  denote the variety generated 
by artinian $*$-regular rings.
Recall that subdirectly irreducible members of $\mathcal{A}$
are representable, whence all members of $\mathcal{A}$
are directly finite.
\begin{enumerate}
\item Is every $*$-regular ring directly finite?
\item Is every  unit-regular
$*$-regular ring a member of $\mathcal{A}$?
\item Is every subdirectly irreducible member of $\mathcal{A}$ unit-regular?
\item Is a subdirectly irreducible $*$-regular  ring $R$ unit-regular
provided so are all $eRe$, $e$ a projection in the minimal ideal?
\end{enumerate}
Observe that  (iv) implies (iii).

\end{document}